\documentclass[pdflatex,sn-mathphys-num]{sn-jnl}

\usepackage{graphicx}%
\usepackage{multirow}%
\usepackage{amsmath,amssymb,amsfonts}%
\usepackage{amsthm}%
\usepackage{mathrsfs}%
\usepackage[title]{appendix}%
\usepackage{textcomp}%
\usepackage{manyfoot}%
\usepackage{listings}%
\usepackage{longtable}
\usepackage{booktabs}\let\cline\cmidrule
\usepackage{makecell}
\usepackage{graphicx}
\usepackage{caption}
\usepackage{subcaption}
\RequirePackage{newtxtext}
\usepackage{soul}
\usepackage{enumitem}

\usepackage{accents}
\usepackage[monochrome]{xcolor}
\usepackage{caption}
\newcommand{\ubar}[1]{\underaccent{\bar}{#1}}
\usepackage[compress,nameinlink]{cleveref}
\crefname{equation}{}{} 
\crefrangelabelformat{subequation}{(#3#1#4-#5\crefstripprefix{#1}#2#6)} 
\crefrangelabelformat{assumption}{Assumptions #3#1#4-#5\crefstripprefix{#1}#2#6} 
\crefrangelabelformat{lemma}{#3#1#4-#5\crefstripprefix{#1}#2#6} 
\crefname{assumption}{Assumption}{} 
\newcommand{\crefdefpart}[2]{%
  \hyperref[#2]{\namecref{#1}~\labelcref*{#1}~\ref*{#2}}%
}
\crefname{enumi}{}{} 
\Crefname{enumi}{}{} 
\newtheorem{lemma}{Lemma}%
\newtheorem{assumption}{Assumption}%

\theoremstyle{thmstyleone}%
\newtheorem{theorem}{Theorem}
\theoremstyle{thmstyletwo}%
\newtheorem{remark}{Remark}%

\theoremstyle{thmstylethree}%
\newcommand{\VV}[1]{\textcolor{black}{#1}}
\begin{document}


\title[Transition Uncertainties in Constrained Markov Decision Models: A Robust Optimization Approach]{Transition Uncertainties in Constrained Markov Decision Models: A Robust Optimization Approach}


\author*[]{\fnm{V} \sur{Varagapriya}}\email{varagapriyav@gmail.com}



\affil[]{\orgdiv{Department of Integrated Systems Engineering}, \orgname{The Ohio State University}, \orgaddress{\city{Columbus}, \postcode{43210}, \state{OH}, \country{USA}}}



\abstract{We examine a constrained Markov decision process under uncertain transition probabilities, with the uncertainty modeled as deviations from observed transition probabilities. We construct the uncertainty set associated with the deviations using polyhedral and second-order cone constraints and employ a robust optimization framework. We demonstrate that each inner optimization problem of the robust model can be equivalently transformed into a second-order cone programming problem. 
Using strong duality arguments, we show that the resulting robust problem can be equivalently reformulated into a \VV{second-order cone programming problem with bilinear constraints.} 
 In the numerical experiments, we study a machine replacement problem and explore potential sources of uncertainty in the transition probabilities. We examine how the optimal values and solutions differ as we vary the feasible region of the uncertainty set, considering only polyhedral constraints and a combination of polyhedral and second-order cone constraints. Furthermore, we analyze the impact of the number of states, the discount factor, and variations in the feasible region of the uncertainty set on the optimal values.     }




\maketitle

\section{Introduction}\label{Introduction}
A sequential decision-making system is a time-evolving model that typically consists of, possibly multiple, decision-makers whose choice of actions determines how the system progresses over time and the incentives received from it. The behavior of the system is influenced by inherent uncertainties in future outcomes. A common modeling technique for such a system with a single decision-maker is a Markov decision process (MDP). A given system can be modeled as an MDP problem in the following way:  at any given time, the system occupies a state from a predetermined set of states. At each state, the decision-maker chooses an action, consequently incurring a running cost. The system then shifts to another state the next time, following an underlying transition probability function, and this process repeats every time. The decision-maker aims to minimize the overall cost incurred; the common overall costs studied in the literature include expected total, average, and discounted costs \cite{puterman}. A few applications of MDP problems include machine replacement and maintenance, inventory management, communication modeling, and queueing control problems \cite{puterman}. Oftentimes, in real-life scenarios, multiple running costs may be incurred for each action chosen. In such cases, the decision-maker aims to minimize one overall cost while ensuring all the other costs are below pre-specified threshold values. For example, a factory owner who owns a machine may have to decide whether to repair it or not at each time, where the time interval could denote months or years. Depending on the action chosen, the owner may incur multiple running costs, such as a working cost and an opportunity cost \cite{cmdp_cost_uncertainty}. In this case, the owner aims to minimize the overall working cost subject to an upper bound constraint on the overall opportunity cost. In general, such a class of MDP problem is called a constrained Markov decision process (CMDP). An MDP problem can be solved using techniques such as dynamic programming algorithms, including value iteration, policy iteration, and their variants, and linear programming (LP)-based reformulation \cite{bertsekas2005dynamic,puterman}. In contrast, the presence of possibly multiple  constraints in a CMDP problem bars us from using dynamic programming algorithms, as a result of which, 
 it is typically solved by reformulating it into an LP problem    \cite{altman}.

 A primary weakness of the solution techniques discussed above is the underlying assumption that the model parameters, namely the running costs and the transition probabilities, are exactly known. These techniques overlook that the values of such parameters are typically obtained from historical data and with prior experience and thus, are subject to errors. Consequently, the optimal solution derived from an MDP/CMDP problem without acknowledging the impending uncertainties may provide us with a degraded solution, i.e., it may not be optimal in a real-life setting \cite{mannor2007bias}. In the previous literature, this drawback has been addressed in an MDP setup using various frameworks. In earlier works such as \cite{satia1973markovian}, an interval-based robust and a Bayesian approach was studied under uncertain transition probabilities, while in \cite{MDP_imprecise_TP},  this problem was considered with uncertainty sets defined by a finite number of linear inequalities. In \cite{givan2000bounded}, the authors described an MDP problem by placing bounds on costs and transition probabilities and consequently derived bounds on the resulting overall cost.  A robust MDP problem under uncertain transition probabilities was also studied in \cite{nilim2005robust,iyengar2005robust}, with separable uncertainty sets for each state-action pair, referred to as $(s,a)$-rectangular sets.   This formulation was generalized to each state in \cite{wiesemann2013robust} and the corresponding uncertainty sets are referred to as $s$-rectangular sets. \VV{The authors in \cite{grand2024convex} considered both these rectangular sets and formulated a convex programming problem by regularizing the robust Bellman operator.} On the other hand, in  \cite{goyal2023robust}, the \VV{authors introduced} a model where a fixed number of factors drive the uncertainty in the transition probabilities.  In \VV{most} of these works, the associated robust MDP problem was solved by constructing specialized dynamic programming-based algorithms. The resulting optimal solution holds against all possible realizations of the uncertainty sets.  The articles reviewed up to this point concentrate on a stationary MDP problem.  For an MDP problem with non-stationary problem data,  an LP approach was derived in \cite{ghate2013linear}, and a simplex algorithm was utilized to achieve the optimal solution for an infinite horizon. Meanwhile, an approximated LP problem was formulated for a finite horizon in \cite{BHATTACHARYA2017570} to reduce the problem size.  
In contrast, by debating the restrictive nature of a robust MDP problem,  a chance-constrained framework was introduced in \cite{erick} under random model parameters. \VV{ Specifically, under random rewards, the authors equivalently derived a second-order cone programming (SOCP) problem, while under random transition probabilities, they constructed approximations.}

 In the context of a CMDP setup, the authors in \cite{zadorojniy2006robustness} investigated conditions under which a minor change in the constraints retains the optimality of the original problem's solution. Meanwhile, in \cite{Kardes}, the author introduced a robust CMDP problem, addressing uncertain running costs that belong to intervals. Since multiple running costs are involved, the dynamic programming algorithms studied in an MDP setup are not applicable. Therefore, the problem was equivalently reformulated into an LP problem, and an optimal policy was derived. This problem was extended in \cite{cmdp_cost_uncertainty} to consider general uncertainty sets, and equivalent convex programming problems were derived. On the other hand, a CMDP problem under random running costs with a known distribution was studied in \cite{varagapriya2022joint}, and SOCP approximations were constructed.   The literature reviewed in a CMDP setup has predominantly addressed uncertain running costs. To the best of our knowledge, there is a notable lack of research addressing uncertain transition probabilities since the resulting uncertain CMDP problem becomes a computationally difficult problem. 
Recently, in \cite{wang2022robust}, the authors investigated this \VV{within a reinforcement learning paradigm, proposing a robust primal-dual algorithm accompanied by its convergence analysis. }
\VV{On the other hand, in \cite{Rankone_TPM}, the authors aimed to provide an exact reformulation of this problem}
by constructing a robust CMDP problem with a special structure on uncertain transition probabilities and equivalently reformulated \VV{it} into a bilinear programming (BP) problem. Specifically, this work assumed that the uncertain transition probability matrix under a given stationary policy has a rank-1 uncertainty structure. However, in real-life scenarios, the uncertainties need not have such a specialized structure; thus, it significantly restricts the problem class that can be studied. \VV{For instance, the machine replacement problem studied in \Cref{Machine replacement problem} of this paper is motivated from \cite{Rankone_TPM}, wherein the authors assumed that the matrix form of uncertainty in transition probabilities has rank 1.} In contrast, we broaden this scope by showing that uncertainty may originate from multiple states and potentially exhibit various dependency structures. Therefore, deriving motivation from \cite{Rankone_TPM}, we express the uncertain transition probabilities as a sum of observed transition probabilities and uncertain parameters. However, we diverge from this work by not imposing any rank-based limitations, i.e., the uncertainty in the transition probabilities could originate from possibly multiple states and may have different values.
The following comparison outlines the novelty of our work relative to \cite{Rankone_TPM}:
\begin{itemize}
    \item We assume that the uncertain parameters belong to a set defined by polyhedral and second-order cone constraints. In addition to the fact that we do not impose any rank-based limitations, we study a broader formulation than existing models on robust MDPs by considering an $s$-rectangular uncertainty set.
\item Due to a generalized uncertainty set considered in this paper, the arguments used to derive a BP problem in \cite{Rankone_TPM} cannot be directly applied here. In \cite{Rankone_TPM}, the authors transformed a linear fractional programming problem from the inner optimization problem of the robust model into an LP problem. In contrast, we introduce a novel approach to defining the decision variables in each inner optimization problem of our robust model and equivalently transform it into an SOCP problem.
\item We employ the strong duality argument of an SOCP problem to show that the resulting robust problem can be equivalently reformulated into  \VV{an SOCP problem with bilinear constraints.} Unlike \cite{Rankone_TPM}, where the authors adopted an occupation measure technique to derive a BP problem, our approach embeds the policy of the robust model directly within the decision vector.
\end{itemize}
\color{black}
  Therefore, this paper significantly generalizes the previous work while also introducing a novel methodology for deriving an equivalent reformulation.

We perform numerical experiments on a machine replacement problem under uncertain transition probabilities \cite{erick}. With a pre-specified time limit, we solve the resulting problem to examine the variation in the optimal values and policies with the feasible region of the uncertainty set. Specifically, we consider only polyhedral constraints with enlarging feasible regions. We then consider a combination of both polyhedral and second-order cone constraints with enlarging feasible regions. Additionally, we analyze the optimal values of the problem under varying states and discount factors.

We organize the paper as follows. In \Cref{Constrained Markov decision processes}, we introduce a classical CMDP problem and discuss its reformulation into an LP problem. In \Cref{Robust CMDP problem}, we construct a robust CMDP problem under uncertain transition probabilities and derive its equivalent reformulation. We analyze the result obtained on a machine replacement problem in \Cref{Machine replacement problem} and conclude the paper in \Cref{Conclusion}.

\section{Constrained Markov decision processes}\label{Constrained Markov decision processes}
We consider a discrete-time infinite horizon CMDP problem and define it using a tuple $\big( S, A, \mathcal{K}, \gamma, p, c, ( d^k )_{k \in \mathbb{K}}, (\xi_k)_{k \in \mathbb{K}} \big)$ \cite{altman}.   The sets $S$ and $A = \displaystyle  \bigcup_{s \in S} A(s)$ denote the finite set of states and actions available at different states, respectively, while the set  $\mathcal{K}$ denotes the set of all state-action pairs, i.e., $\mathcal{K} = \{(s, a) \mid s\in S, a \in A(s)\}$. At time $t=0$, the system starts from a state $s_0$ with a probability $\gamma(s_0)$. Thus, $\gamma = \big( \gamma(s) \big)_{s \in S}$ denotes the initial probability distribution. If the decision-maker chooses an action $a_0 \in A(s_0)$, running costs $c(s_0, a_0)$ and $d^k(s_0, a_0)$,  $k \in \mathbb{K} = \{ 1, 2, \ldots, K \}$ are incurred. Thus, $c = \big( c(s,a) \big)_{(s, a) \in \mathcal{K}}$ and  $d^k = \big( d^k(s,a) \big)_{(s, a) \in \mathcal{K}}$ denote the running cost vectors. At time $t=1$,  the system shifts to the next state $s_1$, following a transition probability $p(s_1 \vert s_0, a_0)$. This process repeats for an infinite horizon.

At a given time $t$  of the system, the action chosen by the decision-maker at a state is described by a decision rule (Section 2.1.4 of \cite{puterman}). It may depend on the history of state-action pairs followed till time $t-1$ along with the state at time $t$. In this case, the decision rule is said to be history-dependent.  Let the history at time $t$ be denoted by $h_t = (s_0, a_0, s_1, a_1,\ldots,s_{t-1}, a_{t-1},s_t)$ and the set of all such histories be denoted by $H_t$. We define a history-dependent decision rule by $f^h_t: H_t \rightarrow \wp(A)$, where $\wp(A)$ denotes the probability distribution on $A$. Alternatively, the decision rule may be Markovian,  i.e., the action chosen at any time $t$ may depend only on the state at time $t$.  This rule becomes stationary when it depends only on the state, i.e., it is independent of both history and time. We denote this rule by $f$.
 The decision rule implemented over time is called a policy. 
We denote a general history-dependent and a stationary policy by $f^h$ and $f$,  and the set of all such policies by $F_{HD}$ and $F_S$, respectively (for simplicity of notation, we denote both the stationary decision rule and the associated policy by $f$).

In this paper, we study the case when the running costs
incurred in the future are discounted with a discount factor  $\alpha \in (0,1)$, i.e., the overall
costs are the expected discounted costs. Denoting a random  state-action pair at a time $t$ by $(\mathbb{X}_t, \mathbb{A}_t)$,  probability measure over the state-action trajectories by $\mathbb{P}_{\gamma}^{f^h} (\cdot)$, and the associated expectation operator by $\mathbb{E}_{\gamma}^{f^h} (\cdot)$, for given initial distribution $\gamma$ and $f^h \in F_{HD}$,  we define the expected discounted  costs associated with $c$ and $d^k$, $k \in \mathbb{K}$, by 
\begin{align*}
     C_{\alpha}(\gamma,f^h) & = (1-\alpha) 
\sum_{t=0}^{\infty}\alpha^{t}\mathbb{E}_{\gamma}^{f^h} (c(\mathbb{X}_t,\mathbb{A}_t))  \nonumber \\
& = \sum_{(s,a) \in \mathcal{K}} g_{\alpha}(\gamma,f^h;s,a)c(s,a),  \\
 D^k_\alpha(\gamma,f^h) & = (1-\alpha)\sum_{t=0}^\infty \alpha^t  \mathbb{E}_{\gamma}^{f^h} (d^k(\mathbb{X}_t,\mathbb{A}_t))  \nonumber  \\ 
    & = \sum_{(s,a) \in \mathcal{K}} g_{\alpha}(\gamma,f^h;s,a)d^k(s,a), \ \forall \ k \in \mathbb{K},  
\end{align*}
where 
$
\displaystyle g_{\alpha}  (\gamma,f^h;s,a)  =  
  (1-\alpha)\sum_{t=0}^\infty \alpha^t \mathbb{P}_{\gamma}^{f^h}(\mathbb{X}_t=s,\mathbb{A}_t=a),$  for all  $(s,a) \in \mathcal{K}. 
 $
The \VV{probability measure,} $g_{\alpha}(\gamma,f^h)$, \VV{also called} the occupation measure\VV{, assigns a probability $g_{\alpha}  (\gamma,f^h;s,a)$ to each $(s,a) \in \mathcal{K}$}  \cite{altman}. The decision-maker aims to derive an optimal policy by minimizing the expected discounted cost associated with $c$ while ensuring the expected discounted costs associated with $d^k$ have pre-specified upper bounds  $\xi_k$, for all $k \in \mathbb{K}$. Thus, we formulate a CMDP problem as
\begin{align}\label{COP_HD_policy}
 \min_{f^h \in F_{HD}} & \  \ C_\alpha(\gamma,f^h) \nonumber \\
  \text{s.t.} & \  \ D^k_\alpha(\gamma,f^h) \leq \xi_k, \ \forall \ k \in \mathbb{K}.
\end{align}
It is well-known that when all the model parameters are exactly known and stationary, i.e., they are independent of time, the above problem can be restricted to the class of stationary policies without loss of optimality (for proof, see Theorem 3.1 of \cite{altman}).  
Furthermore, for a given $f \in F_S$, we can re-define the expected discounted costs corresponding to the cost vectors  $c$ and $d^k$,  $k \in \mathbb{K}$, in matrix form as  
\begin{align}\label{statonry_cost}
     C_{\alpha}(\gamma,f) & = (1-\alpha) \gamma^T \big(I - \alpha P_f \big)^{-1} c_f, \nonumber \\ 
    D^k_\alpha(\gamma,f) & = (1-\alpha) \gamma^T \big(I - \alpha P_f \big)^{-1} d^k_f, \ \forall \ k \in \mathbb{K},
\end{align}
where $(\cdot)^T$ denotes transposition, $c_f$ and $d^k_f$ denote the $\vert S \vert$-dimensional cost vectors under $f$, whose $s^{th}$-components are defined by  $\displaystyle \sum_{a \in A(s)} f(s,a) c(s,a)$ and $\displaystyle \sum_{a \in A(s)} f(s,a) d^k(s,a)$,   respectively. In addition, $P_f$ denotes the $\vert S \vert \times \vert S \vert$-dimensional transition probability matrix under $f$, whose component for a transition from a state $s$ to $s'$ is defined by $\displaystyle \sum_{a \in A(s)} f(s, a) p(s' \vert s, a)$ while $I $ denotes the $\vert S \vert \times \vert S \vert$-dimensional identity matrix. 
Furthermore, it follows from Theorem 3.2 of \cite{altman} that \eqref{COP_HD_policy} under the class of stationary policies can be equivalently reformulated into the following LP problem using a decision vector $\rho \in \mathbb{R}^{|\mathcal{K}|}$, where $\vert \mathcal{K} \vert$ denotes the cardinality of $\mathcal{K}$.
\begin{align}\label{equi_LP}
    \min_{\rho \in \mathbb{R}^{|\mathcal{K}|}} & \  \sum_{(s,a) \in \mathcal{K}} \rho(s,a)c(s,a) \nonumber\\
    \text{s.t.} & \  \sum_{(s,a) \in \mathcal{K}} \rho(s,a)d^k(s,a)\leq \xi_k, \ \forall \ k \in \mathbb{K}, \nonumber \\ 
    & \ \sum_{(s,a) \in \mathcal{K}} \rho(s,a) (\delta(\bar{s}, s)
 -  \alpha  p(\bar{s}|s,a))  
 = (1-\alpha) \gamma(\bar{s}),\ \forall \ \bar{s} \in S, \nonumber \\ 
 & \ \rho(s,a)\geq 0, \ 
 \forall \ (s,a) \in \mathcal{K},
\end{align}
where 
$\delta(\bar{s}, s)$ denotes the Kronecker delta, i.e., $\delta(\bar{s}, s) = 1$ if $\bar{s} = s$, and $0$, otherwise. As a consequence of the above  equivalence, if $\rho^*$ is an optimal solution of \eqref{equi_LP},  we can derive the optimal policy $f^*$ of \eqref{COP_HD_policy} by the relation  $ \displaystyle 
f^*(s,a) = \displaystyle \frac{\rho^*(s,a)}{\displaystyle\sum_{a \in A(s)}\rho^*(s,a)}$,     $(s,a) \in \mathcal{K},  
 $ if $\displaystyle \sum_{a \in A(s)}\rho^*(s,a) >0$. We choose it arbitrarily if $\displaystyle \sum_{a \in A(s)}\rho^*(s,a) =0$.

\section{Robust CMDP problem}\label{Robust CMDP problem}
As discussed in \Cref{Introduction},  the exact values of the model parameters in a CMDP problem need not always be known in real-life scenarios.  Consequently, solving the problem \eqref{COP_HD_policy} without addressing the uncertainties may provide solutions that are not optimal in practice \cite{mannor2007bias}.
Therefore, in this paper, we study a CMDP problem where the transition probabilities are uncertain and stationary while the running costs are known. We denote the vector consisting of these transition probabilities by  $ \big( p(s' \vert s, a) \big)_{
 \substack{ (s, a, s') \in \mathcal{H}   } }$, where $\mathcal{H} = \{ (s,a,s') \mid (s,a) \in \mathcal{K}, s' \in S \}$.

 As in \cite{ben2009robust},  we delineate the true but uncertain transition probabilities by expressing them as a sum of observed transition probabilities and uncertain parameters. While the former values are fixed from historical data and prior experience, the latter values account for the underlying uncertainties in the transition probabilities. 
We  denote the vector associated with former values by $ \big( \bar{p}(s' \vert s, a) \big)_{
 \substack{ (s, a, s') \in \mathcal{H} } }$ and \VV{outline the relation between these three values as: } 
 \VV{
 \begin{enumerate}[label=(P{\arabic*})]
   \item \label{sum_two_values}  $p(s' \vert s, a) = \bar{p}(s' \vert s, a) +  u(s' \vert s, a), $ $\forall \ (s,a,s') \in \mathcal{H}$, 
   \medskip
\item \label{LB_UB_u} $\begin{aligned}[t]   u(s' \vert s, a) \in 
 [ \ubar{u}(s' \vert s, a), \bar{u}(s' \vert s, a) ]; \ \bar{p}(s' \vert s, a) +  \ubar{u}(s' \vert s, a) \geq 0,  \nonumber \\
 \bar{p}(s' \vert s, a) +  \bar{u}(s' \vert s, a) \leq 1, \ \forall \  (s, a, s') \in \mathcal{H},  \nonumber \end{aligned}$ 
 \medskip
\item \label{sum_0_u} $ \displaystyle \sum_{s' \in S} u(s' \vert s, a) = 0, \ \forall \ (s,a) \in \mathcal{K},$ 
\end{enumerate}
}
\noindent
\VV{where the constraints \crefrange{LB_UB_u}{sum_0_u} ensure that the components of both the vectors $\big( p(s' \vert s, a) \big)_{
 \substack{ (s, a, s') \in \mathcal{H}   } }$ and $\big( \bar{p}(s' \vert s, a) \big)_{
 \substack{ (s, a, s') \in \mathcal{H}   } }$ remain transition probabilities. We observe from  \crefrange{sum_two_values}{sum_0_u} that the underlying uncertain parameter is the vector $ \mathfrak{u} = \big(u(s' \vert s, a)\big)_{(s,a,s') \in \mathcal{H}}$.} Consequently,  we assume that it belongs to an uncertainty set $\mathfrak{U}$.

For a given $f^h \in F_{HD}$ and an initial distribution $\gamma$, we denote the resulting uncertain expected discounted costs by $ C_\alpha(\gamma,f^h,\mathfrak{u}) $ and $ D^k_\alpha(\gamma,f^h,\mathfrak{u}) $,  $k \in \mathbb{K}$. 
 Employing a robust optimization framework, wherein a problem with uncertain parameters is solved under their worst-case realization \cite{ben2009robust,zhen2025unified}, we define a robust CMDP problem under uncertain transition probabilities as
\begin{align}\label{robust_cmdp_with_stationary_policy}
\min_{z, f^h \in F_{HD}} & \ z \nonumber \\
   \textnormal{s.t.}  \ \max_{\mathfrak{u} \, \in \, \mathfrak{U}} & \  C_\alpha(\gamma,f^h,\mathfrak{u}) \leq z, \nonumber\\
   \max_{\mathfrak{u} \, \in \, \mathfrak{U}} & \  D^k_\alpha(\gamma,f^h,\mathfrak{u}) \leq \xi_k, \ \forall \ k \in \mathbb{K}.
\end{align} 
\VV{The above formulation ensures that the resulting optimal policy remains optimal under all worst-case scenarios of the expected discounted costs.}
We define the uncertainty set $\mathfrak{U}$ as an intersection of polyhedral and second-order cone constraints\VV{, wherein the polyhedral constraints explicitly model \crefrange{LB_UB_u}{sum_0_u}}. Unlike \cite{Rankone_TPM}, we do not impose rank-based limitations on the uncertain parameters, and we account for the dependencies between the components of  $\mathfrak{u}$. 
Furthermore, by also incorporating second-order cone constraints, we maintain a sufficiently general uncertainty set, accommodating a broad range of uncertainty structures. This includes the case when the uncertain parameters possibly deviate within a pre-specified ellipsoid.
 We define the  set  $\mathfrak{U}$ as
 \begin{align}\label{general_US_combined}
    \mathfrak{U} = \Big\{ \mathfrak{u} =& \big( u(s' \vert s, a) \big)_{
 \substack{ (s, a, s') \in \mathcal{H}} }  \mid \, B(s) \mathfrak{u}(s) - b(s) \leq 0,   \nonumber \\
&\| M(s)^T \mathfrak{u}(s) + m_0(s) \|_2 \leq m_1(s)^T \mathfrak{u}(s) + m_2(s),  \ \forall \ s \in S \Big\}, 
\end{align}
where $\mathfrak{u}(s) = \big( u(s' \vert s, a) \big)_{
 \substack{ (a, s') \in A(s) \times S} }$, $B(s)$ is an $\ell_p(s) \times (\vert A(s) \vert  \vert S \vert)$-dimensional matrix and $b(s)$ is an $\ell_p(s)$-dimensional vector with the associated constraint \VV{including  \crefrange{LB_UB_u}{sum_0_u} along} with possible additional polyhedral constraints. Furthermore, $M(s)$ is a $ (\vert A(s) \vert  \vert S \vert) \times \ell_{sc}(s)$-dimensional matrix and $m_0(s)$, $m_1(s)$, and $m_2(s)$ are vectors of dimensions $\ell_{sc}(s)$, $\vert A(s) \vert  \vert S \vert$, and $1$, respectively. 
 We construct the set $\mathfrak{U}$ to be $s$-rectangular, i.e.,   the vectors $\mathfrak{u}(s)$ are modeled separately, reflecting the separable structure of the vectors $ \big( p(s' \vert s, a) \big)_{
 \substack{ (a, s') \in A(s) \times S} }$, although their components could be dependent, for each $s \in S$. Under mild assumptions, this type of uncertainty set also allows us to leverage the theory of strong duality to derive an equivalent reformulation of the robust problem \eqref{robust_cmdp_with_stationary_policy}.

 For a given $f \in F_S$,  the costs $C_\alpha(\gamma,f,\mathfrak{u})$ and $D^k_\alpha(\gamma,f,\mathfrak{u})$, $k \in \mathbb{K}$, can be expressed in a matrix form similar to \eqref{statonry_cost}. However, this formulation becomes challenging for a general class of policies, $F_{HD}$. 
Consequently, it becomes difficult to exploit the characteristics of transition probabilities to derive a tractable formulation of \eqref{robust_cmdp_with_stationary_policy}. Therefore, we restrict the problem to the class of stationary policies for simplicity, although this results in a loss of optimality.  
 The example given in Appendix A of \cite{Rankone_TPM} considers a robust CMDP problem wherein a single parameter drives the uncertainties in the transition probabilities from each state. The authors solved this problem restricted to the class of stationary policies by equivalently reformulating it into a \VV{bilinear programming (BP)} problem.  Additionally, they constructed a Markov policy that performs better against the optimal stationary policy in terms of the objective value. As a result, they concluded that this restriction is with loss of optimality.  Since the uncertainty set studied in \cite{Rankone_TPM} is a special case of the uncertainty set we study in this paper, the example and the subsequent conclusion of the loss of optimality holds in our case as well.

 Henceforth,  \eqref{robust_cmdp_with_stationary_policy} refers to the problem restricted to the class of stationary policies. 
 For a given feasible vector $(z,f)$ of \eqref{robust_cmdp_with_stationary_policy}, we show in the following two lemmas that its inner optimization problems under \eqref{general_US_combined} can be equivalently reformulated into SOCP problems, given the following assumption.  
 \begin{assumption}\label{assumption_positive_gamma}
    For each $s \in S$, $\gamma(s)>0. $ 
\end{assumption}
A decision-maker may adopt this assumption when historical data indicates the system can start from any state, and it is widely reflected in studies using 
 a uniform distribution \cite{cmdp_cost_uncertainty,erick,wiesemann2013robust}.
 For a fixed $f \in F_S$, we recall the inner optimization problem in the first constraint of \eqref{robust_cmdp_with_stationary_policy} under \eqref{general_US_combined} as
\begin{subequations} \label{Inner_opt_with_c_combined}
 \begin{align} 
  \max_{ \mathfrak{u}} & \  (1-\alpha) \gamma^T {Q}_f  c_f \nonumber \\
  \textnormal{s.t.}  & \ B(s) \mathfrak{u}(s) - b(s) \leq 0,\ \forall \ s \in S, \label{polyhedral_combined} \\ 
  & \ \| M(s)^T \mathfrak{u}(s) + m_0(s) \|_2 \leq m_1(s)^T \mathfrak{u}(s) + m_2(s),  \ \forall \ s \in S, \label{second_order_combined} 
 \end{align}  
 \end{subequations}
 where $Q_f = \big( I-\alpha P_f \big)^{-1}$ and $P_f$ is a $\vert S \vert \times \vert S \vert$-dimensional matrix under $f$ whose component associated with a transition from a state $s$ to $s'$ is defined by $\displaystyle \sum_{a \in A(s) } f(s, a) {p} (s' \vert s,a) = \sum_{a \in A(s) } f(s, a) \big(\bar{p} (s' \vert s,a) + u(s' \vert s,a)\big)$. 
 \begin{lemma}\label{lemma_equivalent_LP_problem_specific_combined}
Let \Cref{assumption_positive_gamma} hold true. For a fixed $f \in F_S$, \eqref{Inner_opt_with_c_combined} is equivalent to the following SOCP problem.
     \begin{subequations}\label{equivalent_LP_problem_specific_combined}
\begin{align}
    \max_{ \mathfrak{w}_c, \mathfrak{z}_c } & \  \mathfrak{w}_c^T c_f   \nonumber \\ 
    \textnormal{s.t.}  & \ B(s) \mathfrak{z}_c(s) - b(s) w_c(s) \leq 0, \ \forall \ s \in S, \label{polyhedral_combined_multiplied} \\
    & \ \| M(s)^T \mathfrak{z}_c(s) + m_0(s) w_c(s) \|_2 \leq m_1(s)^T \mathfrak{z}_c(s) + m_2(s) w_c(s),  \ \forall \ s \in S,  \label{second_order_combined_multiplied} \\
    & \  \mathfrak{w}_c \geq (1-\alpha) \gamma, \label{positive_w_combined} \\
& \ \mathfrak{w}_c^T (I- \alpha \bar{P}_f)  - \alpha \sum_{s \in S} 
\mathfrak{z}_c(s)^T F_f(s)   = (1-\alpha) \gamma^T,  \label{flow_balance_constr_specific_combined}
\end{align}
\end{subequations}
where $\bar{P}_f$ is a $\vert S \vert \times \vert S \vert$-dimensional matrix whose component associated with a transition from a state $s$ to $s'$ is defined by $\displaystyle \sum_{a \in A(s) } f(s, a) \bar{p} (s' \vert s,a)$. Additionally, \VV{$\mathfrak{z}_c = \big(\mathfrak{z}_c(s)\big)_{s \in S}$,} $\mathfrak{z}_c(s) = \big(z_c(s' \vert s, a)\big)_{(a,s') \in A(s)  \times S}$,   $\mathfrak{w}_c = \big(w_c(s)\big)_{s \in S}$, and $F_f(s)$ is a $(\vert A(s) \vert  \vert S \vert) \times \vert S \vert$-dimensional matrix under the policy $f$, for each $s \in S$. \VV{Specifically, it consists of $\vert A(s) \vert$-times vertically concatenated scalar matrices,  such that the $i^{\text{th}}$ scalar matrix has component  $f(s, a^i)$, $a^i \in A(s)$. } 
 \end{lemma}
\begin{proof}
 Let $f \in F_S$ be fixed. 
For a given feasible vector $\mathfrak{u}$ of \eqref{Inner_opt_with_c_combined}, we define a vector 
$(\mathfrak{w}_c, \mathfrak{z}_c )$  with   
 \begin{align} w_c(s) &=  (1-\alpha) \gamma^T Q_f e_s , \ \forall \ s \in S, \label{definition_w_c_combined}\\
 z_c(s' \vert s, a) &= (1-\alpha)\gamma^TQ_f e_s u(s' \vert s,a), \ \forall \ (s,a,s') \in \mathcal{H}. \label{definition_z_c_combined}  
 \end{align} 
 Since   $I $ and $I - \alpha P_f$ are M-matrices, it follows from Theorem 1.8 of  \cite{johnson2011inverse} that $Q_f \geq I$, thus \eqref{positive_w_combined} is  satisfied from \eqref{definition_w_c_combined}. From \Cref{assumption_positive_gamma}, $\mathfrak{w}_c > 0$, thus,  multiplying \eqref{polyhedral_combined} and \eqref{second_order_combined} by $w_c(s)$ and using \eqref{definition_z_c_combined} for each $s \in S$, we obtain \eqref{polyhedral_combined_multiplied} and \eqref{second_order_combined_multiplied}, respectively.    Furthermore, \VV{using a component-wise formulation of $Q_f$} in   \eqref{definition_w_c_combined},  we obtain   
\begin{align}\label{FB_derivation_combined}
    \mathfrak{w}_c^T (I- \alpha \bar{P}_f) e_{s'} - (1-\alpha) \gamma(s') & = \alpha \sum_{s \in S} w_c(s) \left( \sum_{a \in A(s)} f(s,a) u(s' \vert s,a) \right) \nonumber \\ 
    & = \alpha \sum_{s \in S}  \left( \sum_{a \in A(s) } f(s,a) z_c(s' \vert s,a) \right) \nonumber \\ 
    & = \alpha \sum_{s \in S}   \mathfrak{z}_c(s)^T F_f(s) e_{s'},
\end{align}
\VV{where $e_{s'}$ denotes the standard unit vector with 1 at $s'^{\text{th}}$-component and the last equality follows by the definition of $F_f(s)$. Consequently, the matrix formulation of \eqref{FB_derivation_combined}}
 yields \eqref{flow_balance_constr_specific_combined}. Additionally, using  $\eqref{definition_w_c_combined}$, $
     (1-\alpha) \gamma^T {Q}_f  c_f = \mathfrak{w}_c^T c_f$.  Hence, $ ( \mathfrak{w}_c, \mathfrak{z}_c  )$
is a feasible vector of  \eqref{equivalent_LP_problem_specific_combined}.  Conversely,   for a given feasible vector $ ( \mathfrak{w}_c,\mathfrak{z}_c   )$ of \eqref{equivalent_LP_problem_specific_combined}, we define a vector $\mathfrak{u} $ with $u(s' \vert s,a) =\displaystyle  \frac{z_c(s' \vert s,a)}{w_c(s)}$,  $(s,a,s') \in \mathcal{H}$. Using this definition in \eqref{flow_balance_constr_specific_combined} along with arguments analogous to \eqref{FB_derivation_combined}, we obtain  
\begin{align*}
    w_c(s) &=   (1-\alpha) \gamma^T Q_f e_s , \ \forall \ s \in S, 
\end{align*}
thus, $
    \mathfrak{w}_c^T c_f = (1-\alpha) \gamma^T {Q}_f  c_f$. 
 From \Cref{assumption_positive_gamma}, $\mathfrak{w}_c > 0$, thus, dividing \eqref{polyhedral_combined_multiplied} and \eqref{second_order_combined_multiplied} by $w_c(s)$, for each $s \in S$, we obtain \eqref{polyhedral_combined} and \eqref{second_order_combined}, respectively. Hence, $\mathfrak{u}$ is a feasible vector of \eqref{Inner_opt_with_c_combined}. 
\end{proof}
Similar to the preceding lemma, SOCP reformulations hold for all other inner optimization problems in \eqref{robust_cmdp_with_stationary_policy}. For completeness, we summarize these reformulations in the following lemma without proof.
\begin{lemma}\label{lemma_equivalent_LP_problem_specific_combined_d}
Let \Cref{assumption_positive_gamma} hold true. For a fixed $f \in F_S$, the inner optimization problem in the last constraint of \eqref{robust_cmdp_with_stationary_policy} under \eqref{general_US_combined} is equivalent to the following SOCP problem for each $k$.
\begin{align}\label{equivalent_LP_problem_specific_combined_d}
    \max_{ \mathfrak{w}_{d^k}, \mathfrak{z}_{d^k} } & \  \mathfrak{w}_{d^k}^T {d}^k_f   \nonumber \\ 
    \textnormal{s.t.}  & \ B(s) \mathfrak{z}_{d^k}(s) - b(s) w_{d^k}(s) \leq 0, \ \forall \ s \in S, \nonumber \\
    & \ \| M(s)^T \mathfrak{z}_{d^k}(s) + m_0(s) w_{d^k}(s) \|_2 \leq m_1(s)^T \mathfrak{z}_{d^k}(s) + m_2(s) w_{d^k}(s),  \ \forall \ s \in S,  \nonumber \\
    & \  \mathfrak{w}_{d^k} \geq (1-\alpha) \gamma, \nonumber \\
& \ \mathfrak{w}_{d^k}^T (I- \alpha \bar{P}_f)  - \alpha \sum_{s \in S} 
\mathfrak{z}_{d^k}(s)^T F_f(s)   = (1-\alpha) \gamma^T, 
\end{align}
where  $\mathfrak{z}_{d^k}(s) $ and $\mathfrak{w}_{d^k}$ are defined similarly to \Cref{lemma_equivalent_LP_problem_specific_combined}. 
 \end{lemma}
 Using the SOCP reformulations derived in the above lemmas, we provide an equivalent reformulation of \eqref{robust_cmdp_with_stationary_policy} into \VV{an SOCP problem with bilinear constraints under the following assumption.} 
\begin{assumption}\label{Strict_feasibility_assumption}
    The region determined by the second-order cone constraint in the set $\mathfrak{U}$ defined in \eqref{general_US_combined} is strictly feasible. 
\end{assumption}
When the inner optimization problems in \eqref{robust_cmdp_with_stationary_policy} are substituted with equivalently reformulated SOCP problems derived as in \Crefrange{lemma_equivalent_LP_problem_specific_combined}{lemma_equivalent_LP_problem_specific_combined_d}, 
the above assumption ensures that strong duality in these problems holds, allowing us to derive an equivalent reformulation of \eqref{robust_cmdp_with_stationary_policy} as a minimization problem \cite{boyd2004convex}. 
\begin{theorem}\label{Thm:robust_cmdp_BP_SOCP_formulation}
 Let \Crefrange{assumption_positive_gamma}{Strict_feasibility_assumption} hold true.     The robust CMDP problem \eqref{robust_cmdp_with_stationary_policy} with $\mathfrak{U}$ defined by \eqref{general_US_combined}, is equivalent to the following problem. 
 \begin{subequations}\label{robust_cmdp_BP_SOCP_formulation}
    \begin{align}
&\min_{z, f, \mathfrak{d}_c, \mathfrak{d}_{d^k}}  \ z \nonumber \\
   \textnormal{s.t.}   & \  (1-\alpha ) \gamma^T (\varsigma_c  - \eta_c) \leq z,  \\
  &\  c_f(s) + \beta_c(s)^T  b(s) + \mu_c(s)^T m_0(s) + \theta_c(s) m_2(s) + \eta_c(s) - e_s^T (I- \alpha \bar{P}_f)  \varsigma_c = 0, \ \forall \ s \in S,  \label{bilinear_c_constraint} \\ 
  & \ B(s)^T\beta_c(s)  - M(s)\mu_c(s)  - \theta_c(s)m_1(s) - \alpha F_f(s)\varsigma_c  = 0, \ \forall \ s \in S, \label{bilinear_c_constraint_2}  \\ 
  & \ \| \mu_c(s) \|_2 \leq \theta_c(s),  \ \beta_c(s) \geq 0, \ \eta_c \geq 0, \ \forall \ s \in S, \label{SOC_c_constraints} \\
    & \  (1-\alpha ) \gamma^T (\varsigma_{d^k}  - \eta_{d^k}) \leq \xi_k, \ \forall \ k \in \mathbb{K},  \\
  &\ 
  \begin{aligned}[b]
  d^k_f(s) + \beta_{d^k}(s)^T  b(s) + \mu_{d^k}(s)^T m_0(s) + \theta_{d^k}(s) m_2(s) + \eta_{d^k}(s) - e_s^T (I- \alpha \bar{P}_f)  \varsigma_{d^k} = 0, \\  \forall \ s \in S, \ k \in \mathbb{K}, \end{aligned} \label{bilinear_dk_constraint} \\ 
  & \ B(s)^T\beta_{d^k}(s)  - M(s)\mu_{d^k}(s)  - \theta_{d^k}(s)m_1(s) - \alpha F_f(s)\varsigma_{d^k}  = 0, \ \forall \ s \in S, \ k \in \mathbb{K}, 
  \label{bilinear_dk_constraint_2} \\ 
  & \ \| \mu_{d^k}(s) \|_2 \leq \theta_{d^k}(s),  \ \beta_{d^k}(s) \geq 0, \ \eta_{d^k} \geq 0, \ \forall \ s \in S,\ k \in \mathbb{K},  \label{SOC_dk_constraints}
\end{align} 
\end{subequations}
where $c_f(s) = \displaystyle \sum_{a \in A(s)} f(s,a) c(s,a)$. We similarly define $d^k_f(s)$,  $s \in S$. Additionally,  $\mathfrak{d}_c = \big( \big(\beta_c(s)\big)_{s \in S},  \big(\theta_c(s)\big)_{s \in S} , \allowbreak \big(\mu_c(s)\big)_{s \in S},  \eta_c, \varsigma_c\big)$ is  a dual vector such that $\beta_c(s)$ and $\mu_c(s)$ have dimensions $\ell_p(s)$ and $\ell_{sc}(s)$, respectively, $\eta_c$ and $\varsigma_c$ are $\vert S \vert$-dimensional vectors, and $\theta_c(s)$ is a  scalar. We similarly define $\mathfrak{d}_{d^k}$, $k \in \mathbb{K}$. 
\end{theorem}
\begin{proof}
    For a fixed  $z$ and $f \in F_s$, it follows from \Crefrange{lemma_equivalent_LP_problem_specific_combined}{lemma_equivalent_LP_problem_specific_combined_d} that the inner optimization problems of \eqref{robust_cmdp_with_stationary_policy} can be equivalently reformulated into SOCP problems  \eqref{equivalent_LP_problem_specific_combined} and \eqref{equivalent_LP_problem_specific_combined_d}. The dual problem of \eqref{equivalent_LP_problem_specific_combined} is given by  (see \Cref{dual_formulation_first_constr} of \Cref{Deriving_dual_appendix})
    \begin{align} \label{equivalent_formulation_dual_first_constr_within_ppr}
  \min_{\mathfrak{d}_c } &\ (1-\alpha ) \gamma^T (\varsigma_c  - \eta_c) \nonumber \\
  \textnormal{s.t.} & \ \text{\Crefrange{bilinear_c_constraint}{SOC_c_constraints}}. 
\end{align}
For each $k$, we derive the dual problem of \eqref{equivalent_LP_problem_specific_combined_d} similarly.
Furthermore, it follows from \Cref{Strict_feasibility_assumption} that the optimal values of \eqref{equivalent_LP_problem_specific_combined} and \eqref{equivalent_formulation_dual_first_constr_within_ppr} are equal and a similar result holds for  \eqref{equivalent_LP_problem_specific_combined_d}. 
 Therefore, by substituting these dual problems in \eqref{robust_cmdp_with_stationary_policy}, we obtain \eqref{robust_cmdp_BP_SOCP_formulation}. 
\end{proof}
We obtain second-order cone constraints from the first constraints of \eqref{SOC_c_constraints} and \eqref{SOC_dk_constraints} and bilinear constraints from \Crefrange{bilinear_c_constraint}{bilinear_c_constraint_2} and \Crefrange{bilinear_dk_constraint}{bilinear_dk_constraint_2} with bilinear terms   $e_s^T  \bar{P}_f \varsigma_c$, $F_f(s)\varsigma_c$ and $e_s^T \bar{P}_f  \varsigma_{d^k}$, $F_f(s)\varsigma_{d^k}$, respectively. Although \eqref{robust_cmdp_BP_SOCP_formulation} is a non-convex programming problem,  it can be solved using global solvers such as  Gurobi \cite{gurobi}.
\begin{remark}\label{irreducibility_remark}
The \Cref{assumption_positive_gamma} ensures that the vectors $\mathfrak{w}_c > 0$ and $\mathfrak{w}_{d^k} > 0$, $k \in \mathbb{K}$, in \eqref{equivalent_LP_problem_specific_combined} and \eqref{equivalent_LP_problem_specific_combined_d}, respectively. These positivity constraints enable us to derive equivalent SOCP reformulations of the inner optimization problems of \eqref{robust_cmdp_with_stationary_policy} in \Crefrange{lemma_equivalent_LP_problem_specific_combined}{lemma_equivalent_LP_problem_specific_combined_d}. Notably, these constraints also hold under an alternative  assumption that $\gamma^T \big(I - \alpha P_{\min}\big)^{-1} >0$, 
     where $P_{\min}$ is a $\vert S \vert \times \vert S \vert$-dimensional matrix whose component associated with a transition from a state $s$ to $s'$ is defined by $ \displaystyle 
 p_{\min}(s' \vert s) =
\min_{a \in A(s)}\big( \bar{p}(s' \vert s, a) +  \ubar{u}(s' \vert s, a)\big)$,  for all  $s,s' \in S $. Similar to Lemma 3.8 and Remark 3.9 of \cite{Rankone_TPM}, a sufficient condition under which this assumption holds true is  $e_s^T \big( I- \alpha P_{\min} \big)^{-1} > 0$, for all $s\in S$. This condition implies that under any stationary deterministic policy, the Markov chain of the system is irreducible. In \Cref{lemma_alternative_equivalent_LP_problem_specific_combined} of \Cref{Alternative_equivalence_appendix}, we show that the alternative assumption also leads to SOCP reformulations of the inner optimization problems in \eqref{robust_cmdp_with_stationary_policy}.  We derive the consequent equivalent reformulation of \eqref{robust_cmdp_with_stationary_policy} as a minimization problem in \Cref{Theorem_duality_bilinear_socp} of \Cref{Alternative_equivalence_appendix}.    
\end{remark}

\section{Machine replacement problem}\label{Machine replacement problem}
We perform numerical experiments on a variant of a machine replacement problem described in Section 5.2 of \cite{erick}. The problem aims to derive an optimal policy to decide whether to repair a machine or not at its various possible states. In general, determining the optimal value and policy before a system initiates can offer valuable insights into the structure of the system. 
Moreover, by incorporating potential uncertainties in the model parameters upfront, we can derive a robust policy that remains effective across all possible parameter variations.
 We describe the system as follows: Consider a factory owner who owns a number of machines, and each machine can be in one of a finite set of states that represents its working state. The set of states is denoted by $S = \{s^1, s^2, \ldots, s^n\}$, where  $s^1$ to $s^{n-2}$ represent the increasing age of the machine, while  $s^{n-1}$ and $s^n$ represent the minor and the major repaired states, respectively.  At each state, the owner chooses an action from the set $A= \{a^1: `\text{do not repair'}, a^2: `\text{repair'}\}$.  According to an underlying transition probability, the machine then moves to another state, and the process continues for an infinite horizon. Motivated from  Figure 3 of \cite{erick}, we assume that the observed transition probabilities in a matrix form with $n=7$ are given as in \eqref{observed_TPMs}. For a higher number of states, the transition probabilities are similarly fixed; for any $n \geq 7$, the transition probabilities from the set of states $\{ s^2, s^3, \ldots, s^{n-4} \}$ are the same as the transition probabilities from $s^2$ while the transition probabilities from $s^1$, $s^{n-3}$, $s^{n-2}$, $s^{n-1}$, and $s^n$ are the same as the transition probabilities from $s^1$, $s^4$, $s^5$, $s^6$, and $s^7$, respectively.

 \small{
    \begin{align}
     \bar{P}(a^1)  & = \hspace{0.2cm}  \bordermatrix{     
            & s^1 & s^2 & s^3 & s^4 & s^5 & s^6 & s^7 \cr
    s^1 & 0.3 & 0.6 & 0 & 0 & 0 & 0.1 & 0 \cr
    s^2 & 0.05 & 0.2 & 0.6 & 0.05 & 0 & 0.1 & 0 \cr
    s^3 & 0 & 0.05 & 0.2 & 0.6 & 0.05 & 0.1 & 0 \cr
    s^4 & 0 & 0 & 0.1 & 0.2 & 0.6 & 0.1 & 0 \cr
    s^5 & 0 & 0 & 0 & 0.1 & 0.8 & 0 & 0.1 \cr
    s^6 & 0.8 & 0 & 0 & 0 & 0 & 0.2 & 0 \cr
    s^7 & 0 & 0 & 0 & 0 & 0.1 & 0.1 & 0.8 \cr
}, \nonumber \\
\bar{P}(a^2) & = \hspace{0.2cm}
\bordermatrix{     
            & s^1 & s^2 & s^3 & s^4 & s^5 & s^6 & s^7 \cr
    s^1 & 0.7 & 0.3 & 0 & 0 & 0 & 0 & 0 \cr
    s^2 & 0.7 & 0.15 & 0.05 & 0.05 & 0 & 0.05 & 0 \cr
    s^3 & 0 & 0.7 & 0.15 & 0.05 & 0.05 & 0.05 & 0 \cr
    s^4 & 0 & 0 & 0.7 & 0.2 & 0.05 & 0.05 & 0 \cr
    s^5 & 0 & 0 & 0 & 0.7 & 0.25 & 0 & 0.05 \cr
    s^6 & 0.1 & 0 & 0 & 0 & 0 & 0.9 & 0 \cr
    s^7 & 0.05 & 0 & 0 & 0 & 0 & 0.6 & 0.35 \cr
    }. \label{observed_TPMs}
    \end{align}}
    \normalsize
For each action chosen at a state, the owner incurs two running costs, namely, working cost and opportunity cost. The former cost is incurred according to the working state of the machine, while the latter cost is incurred due to the production of inferior quality goods when the owner chooses $a^1$ and due to a production loss when the owner chooses $a^2$ \cite{cmdp_cost_uncertainty}. The owner aims to derive an optimal policy by minimizing the expected discounted working cost while ensuring the expected discounted opportunity cost stays below an upper bound value. Since all machines are identical, the same policy can be applied to all the machines; hence, we consider a single machine.

From \eqref{observed_TPMs}, we observe that when the owner chooses $a^1$, the machine moves to its next state with a higher probability when its present state belongs to the set $\{s^1, s^2, \ldots, s^{n-3}\}$. At states $s^{n-2}$ and $s^{n}$, it stays in the same state with a higher probability since they represent poor and major repaired states, respectively. 
While these observations are consistent with an expected scenario,
the instance where the machine moves from $s^{n-1}$ to $s^1$ with a higher probability is contradictory. Since the owner does not repair the machine, it is likely to remain in the same state or move to $s^n$ with a higher probability. On the other hand, when the owner chooses $a^2$,   the machine moves to its previous state with a higher probability when its present state belongs to the set $\{s^2, s^3, \ldots, s^{n-2}, s^n\}$. However, contrary to the observation at $s^{n-1}$, the machine may move to a better state with a higher probability since the owner repairs the machine.  
Moreover, as observed at $s^n$, it may move to $s^1$ from all other states with a positive probability. 
Consequently, the observed transition probabilities obtained in \eqref{observed_TPMs} may not represent the true values, indicating possible deviations in the true but uncertain transition probabilities. Thus, we define a general uncertainty set $\mathfrak{U}$ in \eqref{general_US_combined} as
\begin{subequations} \label{Uncertainty_set_MR}
\begin{align} 
\mathfrak{U} =&   
     \Big\{ \mathfrak{u} = \big( u(s^j \vert s^i, a^m) \big)_{
 \substack{ (s^i, a^m, s^j) \in \mathcal{H}} }  \vert \nonumber \\
  & \begin{aligned}[b]
   u(s^1 \vert s^{n-1},a^1 ) \in [-0.5, 0], \ u(s^{n-1} \vert s^{n-1},a^1 ) \in [-0.1, 0.1], \ u(s^{n} \vert s^{n-1},a^1 ) \in \\ [0, 0.6],  
   u(s^1 \vert s^{n-1},a^2 ) \in [0, 0.7], \  u(s^{n-1} \vert s^{n-1},a^2 ) \in [-0.8, 0], \\ 
   u(s^1 \vert s^{n-1}, a^1) + 2u(s^{n-1} \vert s^{n-1}, a^1) + 5u(s^1 \vert s^{n-1}, a^2) \leq 1,
   \end{aligned} \label{constr_n-1_state}
      \\
  &  \begin{aligned}[t]
       u(s^{1} \vert s^{i},a^2 ) \in [0, 0.7], \ u(s^{i-1} \vert s^{i},a^2 ) \in [-0.4, 0.1], \ \forall \ i \in \{ 3,4,\ldots,n-2\}, 
   \end{aligned}  \label{constr_3_n_2_state} \\ 
& \begin{aligned}[b]
u(s^j \vert s^i ,a^m) \in \big[ -\sigma  \bar{p}(s^j \vert s^i ,a^m), \sigma( 1- \bar{p}(s^j \vert s^i ,a^m)) \big], \ \forall \ (s^i, a^m, s^j) \in \\
\mathcal{H} \backslash \text{\crefrange{constr_n-1_state}{constr_3_n_2_state}}
\end{aligned} \label{constr_all_states_scale}, \\
& \begin{aligned}[t]
\sum_{j=1}^n u (s^j \vert s^i, a^m) = 0, \ \forall \ (s^i, a^m) \in \mathcal{K}, 
\end{aligned} \label{constr_sum_0} \\
& \begin{aligned}[t]
\|  \mathfrak{u}(s^i)   \|_2 \leq    m_2(s^i),  \ \forall \ s^i \in S
\end{aligned}     
\Big\},   \label{Norm_constraint}
\end{align}
\end{subequations}
where $\sigma \in [0,1]$ is a scale parameter that controls the deviation of transition probabilities from their observed values. We analyze the sensitivity of the optimal values to variations in the constraints of the robust CMDP problem \eqref{robust_cmdp_with_stationary_policy}. We solve all our problems using Gurobi solver \cite{gurobi} of YALMIP toolbox \cite{Lofberg2004} in MATLAB on an Intel(R) 64-bit Core(TM) i5-1240P CPU @ 1.70GHz with 16.0 GB RAM machine.

{\scriptsize\tabcolsep = 1.0pt 
\renewcommand\arraystretch{1.5} 
\begin{longtable}{@{\extracolsep{1.5pt}}p{0.8cm}p{0.8cm}p{3cm}p{1.3cm}p{1cm}p{1cm}p{1cm}p{0.7cm}p{0.7cm}p{0.7cm}p{0.7cm}@{}}
\caption{Optimal values/gaps and policies under $n = 7$.}\\
\hline
$\sigma$ & $m_2(s)$ & $z^*$/[LB, UB] & Gap($\%$)& \multicolumn{7}{c}{Probability of Repair} \\ 
 \cline{5-11}
& &  &  & $s^1$ & $s^2$ & $s^3$ & $s^4$ & $s^5$ & $s^6$ & $s^7$ \\
  \hline
\endfirsthead
\caption{Optimal values/gaps and policies under $n = 7$.}\\
\hline
$\sigma$ & $m_2(s)$ & $z^*$/[LB, UB] & Gap($\%$)& \multicolumn{7}{c}{Probability of Repair} \\
 \cline{5-11}
& &  &  & $s^1$ & $s^2$ & $s^3$ & $s^4$ & $s^5$ & $s^6$ & $s^7$ \\ 
  \hline
  \endhead
\midrule
    \multicolumn{11}{r}{\footnotesize\itshape Continued on the next page}
\endfoot
\endlastfoot
 0&	-&	84.9511&	0&	0.7649&	0 &	1 &	1 &	1 &	1 &	1 
 \\ 
 \hline 
 0.01&	-&	92.7133&	0&	0.7082&	0 &	1 &	1 &	1 &	1 &	1  \\ 
 &0.01&	90.0318&	0&	0.6448&	0.1445&	1 &	1 &	1 &	1 &	1  \\
 &0.1&	92.7133&	0&	0.7082&	0 &	1 &	1 &	1 &	1 &	1  \\
 &0.5&	[86.8082, 92.7144]&	6.3703&	0.7081&	0 &	1 &	1 &	1 &	1 &	1  \\ 
 \hline 
 0.03&	-&	107.9344&	0&	0 &	0.8307&	1 &	1 &	1 &	1 &	1  \\
 &0.01&	90.2866&	0&	0.6123&	0.1952&	1 &	1 &	1 &	1 &	1  \\
 &0.1&	[107.5718, 107.7116]&	0.1298&	0.5591&	0.062&	1 &	1 &	1 &	1 &	1  \\
 &0.5&	[96.7416, 107.9377]&	10.3727&	0.0004&	0.83 &	1 &	1 &	1 &	1 &	1  \\ 
 \hline
 0.05&	-&	122.5219&	0&	0 &	0.6745&	1 &	1 &	1 &	1 &	1  \\
 &0.01&	[73.7853, 90.3796]&	18.3607&	0.6088&	0.1994&	1 &	1 &	1 &	1 &	1  \\
 &0.1&	[105.0759, 120.2469]&	12.6166&	0.2865&	0.3111&	1 &	1 &	1 &	1 &	1  \\ 
 &0.5&	[102.6804, 122.5263]&	16.1972&	0 &	0.6743&	1 &	1 &	1 &	1 &	1  \\
 \hline
 0.07&	-&	137.5278&	0&	0 &	0.5075&	1 &	1 &	1 &	1 &	1  \\ 
 &0.01&	[78.3563, 90.4313]&	13.3527&	0.607&	0.2016&	1 &	1 &	1 &	1 &	1  \\
 &0.1&	[110.0556, 128.3822]&	14.275&	0.3201&	0.159&	1 &	1 &	1 &	1 &	1  \\ 
 &0.5&	[113.6387, 137.531]&	17.3723&	0 &	0.5073&	0.9999&	1 &	1 &	1 &	1  \\ 
 \hline 
 0.1&	-&	[157.8017, 160.0454]&	1.4019&	0 &	0.2223&	0.1196&	1 &	1 &	1 &	1  \\ 
 &0.01&	[77.0855, 90.4388]&	14.765&	0.6059&	0.2031&	1 &	1 &	1 &	1 &	1  \\ 
 &0.1&	[114.9212, 132.5782]&	13.3182&	0.3365&	0.0889&	1 &	1 &	1 &	1 &	1  \\ 
 &0.5&	[158.9444, 160.0507]&	0.6912&	0 &	0.2225&	0.162 &	1 &	1 &	1 &	1  \\ 
 \hline
 0.3 &	- & Inf& -&-&-&-&-&-&-&- \\
 &0.01&	[84.6690, 90.4405]&	6.3816&	0.6058&	0.2032&	1&	1&	1&	1&	1\\
 &0.1&	[123.9431, 136.9922]&	9.5255&	0.3249&	0.0355&	1&	1&	1&	1&	1\\
 &0.5 & Inf& -&-&-&-&-&-&-&- \\
\hline
\footnotemark
 \label{table_CPU_time_BP}
 \end{longtable}
 \footnotetext{`Inf' indicates an infeasible problem; $\text{Gap}(\%) = \displaystyle \frac{\vert \text{LB-UB} \vert}{\vert \text{UB} \vert} \times 100. $ }
}
\begin{figure}[hbt!]
    \centering
    \begin{subfigure}[c]{0.5\linewidth}
        \centering
    \includegraphics[scale=0.28]{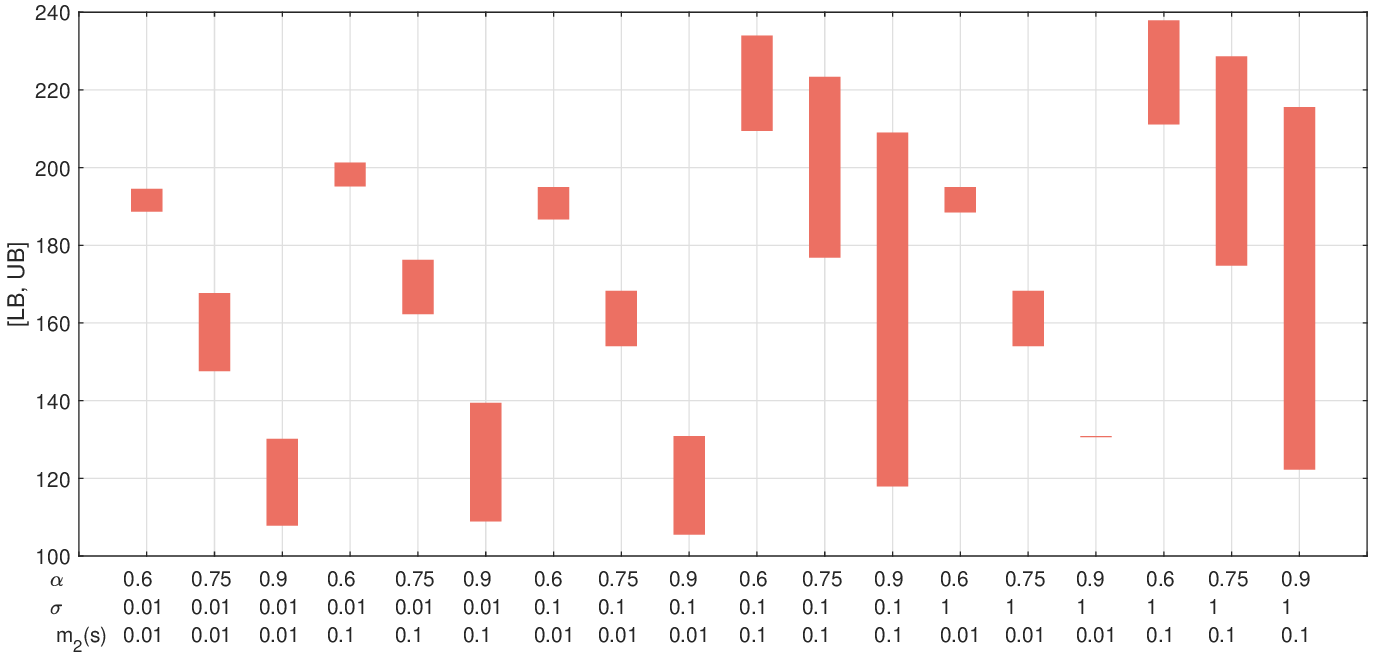}
    \caption{\footnotesize $\vert S \vert = 10$.  }
\label{Bounds_Plot_10_States}
 \end{subfigure}%
    \begin{subfigure}[c]{0.5\textwidth} 
    \centering
      \includegraphics[scale=0.28]{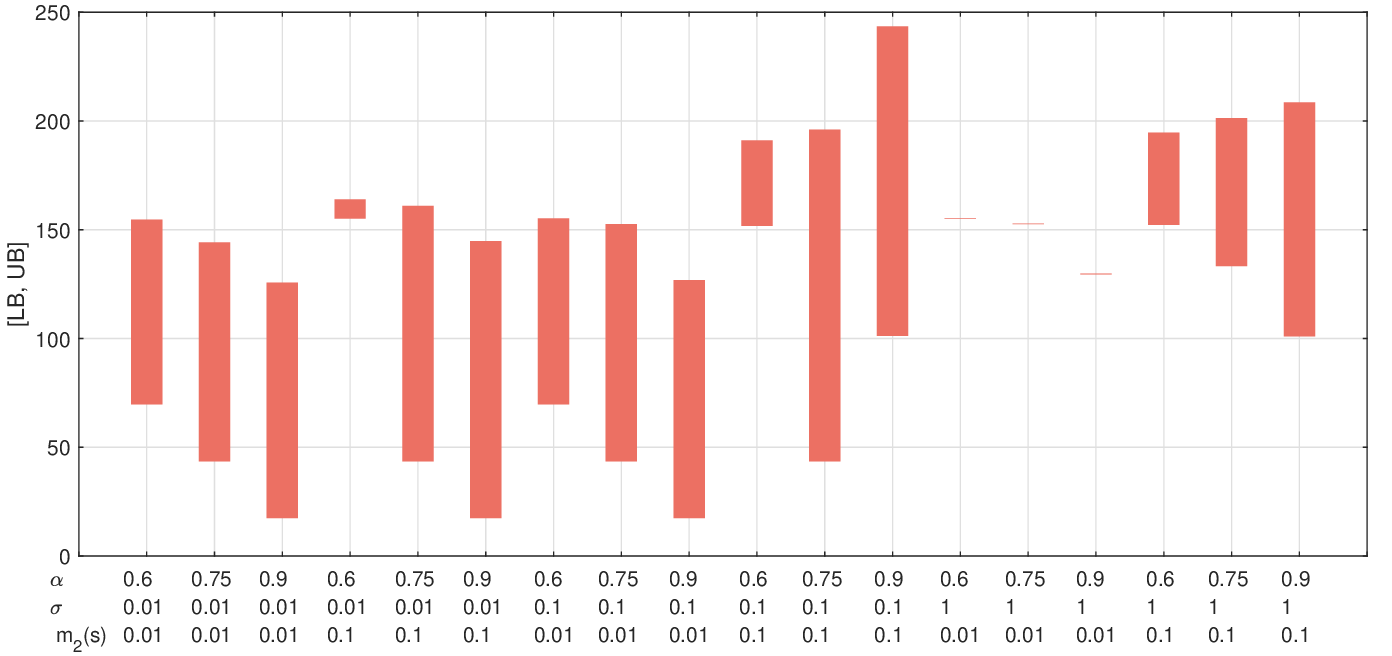}
    \caption{\footnotesize $\vert S \vert = 25$.  }
    \label{Bounds_Plot_25_States}
    \end{subfigure}
   \vspace*{-1mm}
    \caption{Lower and upper bounds of optimal values.}
   \label{Bounds_Plot}
\end{figure}
As in \cite{Rankone_TPM}, we fix $n=7$,  $\alpha=0.6$, and assume the system starts at $s^1$. Due to the latter assumption, we solve the alternative equivalent problem \eqref{robust_cmdp_BP_SOCP_formulation_alternate_reformulation} of the robust CMDP problem \eqref{robust_cmdp_with_stationary_policy} under a wall-clock time limit of $14400$ seconds. The running cost vectors are fixed as:
\begin{align*}
  & \big( c(s) \big)_{s \in S}  = (61.08,
62.17, 144.44, 174.36, 800, 300, 600), \\
 & \big( d^1(s,a^1) \big)_{s \in S}  = (113.64, 154.73, 173.2, 191.32,600,800,900
), \\ 
 & \big( d^1(s,a^2) \big)_{s \in S}  = (179.33,269.52,189.51,258.9,200,250,350).
\end{align*}
We fix $\xi_1=170$ and vary the values of $\sigma$ and $\big( m_2(s) \big)_{s \in S}$,  keeping the components of $m_2(s)$ constant. 
We summarize the optimal values (or the gap between the lower and upper bounds of the optimal values, denoted by LB and UB, respectively) and the associated policies (or the feasible policies) in \Cref{table_CPU_time_BP}. The first row of the table considers the constraints \eqref{constr_n-1_state} and \eqref{constr_sum_0} of $\mathfrak{U}$ and it is identical to the problem considered in Section 4.1 of \cite{Rankone_TPM}. The other rows with no value of  $m_2(s)$ consider the constraints \crefrange{constr_n-1_state}{constr_sum_0}, while all the remaining rows consider the constraints \crefrange{constr_n-1_state}{Norm_constraint}. We observe from the rows with 
 no value of $m_2(s)$ that the optimal values increase with $\sigma$ since the feasible region of each inner optimization problem in \eqref{robust_cmdp_with_stationary_policy} enlarges with $\sigma$. However, for other values of $m_2(s)$, this remains inconclusive due to the time limit.  
On the other hand, for $\sigma \in \{0.1, 0.3\}$, the values of LB and UB indicate that the optimal values are highest with no value of $m_2(s)$,  decrease at $m_2(s)=0.01$ and then increase with $m_2(s)$. This is because when $m_2(s)$ moves to $0.01$, the constraints \eqref{Norm_constraint} are introduced in addition to the existing constraints \crefrange{constr_n-1_state}{constr_sum_0} in $\mathfrak{U}$. As $m_2(s)$ increases, the feasible region of each inner optimization problem in \eqref{robust_cmdp_with_stationary_policy} enlarges, thereby increasing the optimal values. This result is inconclusive for other values of $\sigma$ due to the time limit. Irrespective of the policy being optimal or feasible, we conclude that repairs must be made in the $4$ states.

We increase the value of $\vert S \vert$ and vary the values of $\alpha$, $\sigma$, and $\big( m_2(s) \big)_{s \in S}$, keeping the components of $m_2(s)$ constant, and assume that $\gamma$ is uniformly distributed. We solve \eqref{robust_cmdp_BP_SOCP_formulation} under a wall-clock
 time limit of  $7200$ seconds.  We randomly generate increasing values of $\big( c(s) \big)_{s \in S}$, $\big( d^1(s,a^1) \big)_{s \in S} $, and $\big( d^1(s,a^2) \big)_{s \in S}$ at the first $n-3$ states, with values from the intervals $(60, 180)$, $(110, 200)$, and $(175, 260)$, respectively.   The values at the last $3$ states are identical to those at the last $3$ states when $n=7$. We fix $\xi_1=300$ and summarize the  lower and upper bounds of the optimal values for $\vert S \vert \in \{10, 25\}$ in \Cref{Bounds_Plot}.   
 We obtain optimal values for a few instances at $(\sigma, m_2(s)) = (1, 0.01)$. \VV{For other instances, the gap ranges between $2.8794$ and $43.5125$ for $\vert S \vert = 10$, with 7 instances exhibiting a gap below 10. When $\vert S \vert = 25$, the gap ranges between $5.2907$ and  $87.9458$ with only 1 instance exhibiting a gap below 10. This indicates reduced solver performance with an increasing value of $\vert S \vert$. } While in some instances, the values of LB and UB indicate that the optimal values decrease with $\alpha$ and increase with $m_2(s)$, in other instances, this remains inconclusive due to the time limit.

\section{Conclusion}\label{Conclusion}
We investigate a CMDP problem characterized by uncertain transition probabilities and consider a robust optimization framework.  We assume that these probabilities can be expressed as a sum of observed transition probabilities and uncertain parameters and construct a generalized set for the uncertain parameters consisting of polyhedral and second-order cone constraints. We present a novel approach to equivalently transform the inner optimization problems of our robust model into SOCP problems, thereby reformulating the overall robust CMDP problem into \VV{an SOCP problem with bilinear constraints, allowing us} to solve the problem using Gurobi. 
In the numerical experiments, we observe that some instances yield optimal values while others provide lower and upper bounds within a pre-specified time limit\VV{, with a relatively higher gap for a larger number of states.} \Cref{table_CPU_time_BP} and \Cref{Bounds_Plot} do not indicate a correlation between the size of the feasible region and whether we obtain optimal values or only bounds. While some cases yield only bounds,  they serve as a valuable starting point for comprehending the system and developing specialized algorithms to efficiently solve it. \VV{  We identify this as a future research direction, along with investigating more general uncertainty sets.}

\bmhead{Acknowledgements}
The author would like to thank Dr. Vikas Vikram Singh, Department of Mathematics, IIT Delhi, for providing constructive feedback on the manuscript.

\section*{Declarations}
No funding was received for conducting this study. The author has no relevant financial or non-financial interests to disclose.
\color{black}
\appendix

\section{Dual formulations of inner optimization problems }\label{Deriving_dual_appendix}
 In \Crefrange{lemma_equivalent_LP_problem_specific_combined}{lemma_equivalent_LP_problem_specific_combined_d}, we showed that for a given feasible vector $(z,f)$ of \eqref{robust_cmdp_with_stationary_policy}, its inner optimization problems under \eqref{general_US_combined} can be equivalently reformulated into SOCP problems. We derive the dual formulations of these problems in the following two lemmas. For a fixed $f \in F_S$, the problem \eqref{equivalent_LP_problem_specific_combined} can be written  as
\begin{align}\label{equivalent_LP_problem_specific_primal}
    \max_{ \mathfrak{w}_c, \mathfrak{z}_c, q_c(s), r_c(s) } & \  \mathfrak{w}_c^T c_f  \nonumber \\ 
    \textnormal{s.t.} & \ B(s) \mathfrak{z}_c(s) - b(s) w_c(s) \leq 0, \ \forall \ s \in S, \nonumber \\
    & \ \| q_c(s) \|_2 \leq r_c(s),  \ \forall \ s \in S,  \nonumber \\
    & \ M(s)^T \mathfrak{z}_c(s) + m_0(s) w_c(s) = q_c(s), \forall \ s \in S, \nonumber \\
    & \ m_1(s)^T \mathfrak{z}_c(s) + m_2(s) w_c(s) = r_c(s), \forall \ s \in S, \nonumber \\
    & \  \mathfrak{w}_c \geq (1-\alpha) \gamma, \nonumber \\
& \ \mathfrak{w}_c^T (I- \alpha \bar{P}_f)  - \alpha \sum_{s \in S} 
\mathfrak{z}_c(s)^T F_f(s)   = (1-\alpha) \gamma^T. 
\end{align}
\begin{lemma}\label{dual_formulation_first_constr}
    Let \Cref{Strict_feasibility_assumption} hold true. For a fixed $f \in F_S$, \eqref{equivalent_LP_problem_specific_primal} is equivalent to the following SOCP problem. 
    \begin{align}\label{equivalent_formulation_dual_first_constr}
  \min_{\mathfrak{d}_c } &\ (1-\alpha ) \gamma^T (\varsigma_c  - \eta_c) \nonumber \\
  \textnormal{s.t.} &\  \textnormal{\Crefrange{bilinear_c_constraint}{SOC_c_constraints}}.
\end{align}
\end{lemma}
\begin{proof}
 For a fixed $f \in F_S$, let $ \big( \big(\beta_c(s)\big)_{s \in S},  \big(\theta_c(s)\big)_{s \in S} , \big(\mu_c(s)\big)_{s \in S},  \big(\lambda_c(s)\big)_{s \in S},  \eta_c, \varsigma_c\big)$ be the dual vector associated with \eqref{equivalent_LP_problem_specific_primal}. The associated Lagrangian function is given by \color{black}
\begin{align*}
    L(\mathfrak{w}_c, \mathfrak{z}_c, & q_c(s), r_c(s), \mathfrak{d}_c , \lambda_c(s)) \\
   & \begin{aligned}[t]
    = \mathfrak{w}_c^T c_f + \sum_{s \in S} \beta_c(s)^T \big(b(s) w_c(s) -  B(s) \mathfrak{z}_c(s) \big) + \sum_{s \in S} \theta_c(s) \big( r_c(s) - 
    \| q_c(s) \|_2 \big) \\
    + \sum_{s \in S} \mu_c(s)^T \big( M(s)^T \mathfrak{z}_c(s) +   m_0(s) w_c(s) - q_c(s) \big)
    + \sum_{s \in S} \lambda_c(s) \big( m_1(s)^T \mathfrak{z}_c(s) \\
    +  m_2(s) w_c(s) -  r_c(s) \big) + \eta_c^T \big( \mathfrak{w}_c - (1- 
    \alpha) \gamma \big) +  \big( (1-\alpha) \gamma^T \\
    + \alpha \sum_{s \in S} \mathfrak{z}_c(s)^T F_f(s)  - \mathfrak{w}_c^T (I- \alpha \bar{P}_f) \big)\varsigma_c
    \end{aligned}\\
    & =
    \begin{aligned}[t]
    \sum_{s \in S} w_c(s)\big( c_f(s) + \beta_c(s)^T  b(s) + \mu_c(s)^T m_0(s) + \lambda_c(s) m_2(s) + \eta_c(s) - \\
     e_s^T (I- \alpha \bar{P}_f)  \varsigma_c \big)
    +  \sum_{s \in S} \big( -\beta_c(s)^T B(s) + \mu_c(s)^T M(s)^T + \lambda_c(s)m_1(s)^T \\
    + \alpha \varsigma_c^T F_f(s)^T \big) \mathfrak{z}_c(s)  
    + \sum_{s \in S} \big( -\theta_c(s) \| q_c(s) \|_2 - 
    \mu_c(s)^T q_c(s) \big) 
    + \sum_{s \in S} \big(  \theta_c(s) \\ 
    - \lambda_c(s) \big)r_c(s) 
    + (1-\alpha ) \gamma^T (\varsigma_c  - \eta_c).
    \end{aligned}
\end{align*}
Consequently, the associated Lagrange dual function is given by 
\begin{align*}
\mathcal{L}(\mathfrak{d}_c , \lambda_c(s)) = \max_{\mathfrak{w}_c, \mathfrak{z}_c, q_c(s), r_c(s)}  L(\mathfrak{w}_c, \mathfrak{z}_c, & q_c(s), r_c(s), \mathfrak{d}_c , \lambda_c(s)).     
\end{align*}
Using the facts that a linear function is bounded above only when it is the zero function and for each $s \in S$,
\begin{align*}
    \max_{q_c(s)} \big( -\theta_c(s) \| q_c(s) \|_2 - 
    \mu_c(s)^T q_c(s) \big) = \begin{cases}
        0, & \| \mu_c(s) \|_2 \leq \theta_c(s), \\ 
        +\infty, & \text{otherwise},
    \end{cases}
\end{align*}
\color{black}
 we obtain the dual problem as   
\begin{align*}
  \min_{\mathfrak{d}_c, \lambda_c(s) } &\ (1-\alpha ) \gamma^T (\varsigma_c  - \eta_c)\\
  \textnormal{s.t.} &\  c_f(s) + \beta_c(s)^T  b(s) + \mu_c(s)^T m_0(s) + \lambda_c(s) m_2(s) + \eta_c(s) - e_s^T (I- \alpha \bar{P}_f)  \varsigma_c = 0, \ \forall \ s \in S, \\ 
  & \ -\beta_c(s)^T B(s) + \mu_c(s)^T M(s)^T + \lambda_c(s)m_1(s)^T + \alpha \varsigma_c^T F_f(s)^T = 0, \ \forall \ s \in S, \\ 
  & \ \| \mu_c(s) \|_2 \leq \theta_c(s), \ \forall \ s \in S, \\ 
  & \ \theta_c(s) = \lambda_c(s), \ \forall \ s \in S, \\
  & \ \beta_c(s) \geq 0,  \eta_c \geq 0, \ \forall \ s \in S.
\end{align*}
The above problem can be equivalently written as in \eqref{equivalent_formulation_dual_first_constr}. 
\end{proof}
 Similar to the preceding lemma, dual formulations hold for all other inner optimization problems in \eqref{robust_cmdp_with_stationary_policy}. For completeness, we summarize this formulation for each $k$ in the following lemma without proof.
\begin{lemma}
    Let \Cref{Strict_feasibility_assumption} hold true. For a fixed $f \in F_S$, \eqref{equivalent_LP_problem_specific_combined_d} is equivalent to the following SOCP problem. 
    \begin{align*} 
  \min_{\mathfrak{d}_{d^k} } &\ (1-\alpha ) \gamma^T (\varsigma_{d^k}  - \eta_{d^k}) \nonumber \\
  \textnormal{s.t.} &\  \textnormal{\Crefrange{bilinear_dk_constraint}{SOC_dk_constraints}}.
\end{align*}
\end{lemma}
\section{Alternative reformulation of \eqref{robust_cmdp_with_stationary_policy}}\label{Alternative_equivalence_appendix}
As stated in \Cref{irreducibility_remark}, we consider the following assumption. 
\begin{assumption}\label{irreducible_assumption}
    The inequality $\gamma^T \big(I - \alpha P_{\min}\big)^{-1} >0$ holds true.
\end{assumption}
\begin{lemma}\label{lemma_alternative_equivalent_LP_problem_specific_combined}
    Let \Cref{irreducible_assumption} hold true. For a fixed $f \in F_S$, \eqref{Inner_opt_with_c_combined} is equivalent to the following SOCP problem.   \begin{subequations}\label{equivalent_SOCP_problem_specific_combined_alternative_formulation}
\begin{align}
    \max_{ \mathfrak{w}_c, \mathfrak{z}_c } & \  \mathfrak{w}_c^T c_f   \nonumber \\ 
    \textnormal{s.t.}  
    & \  \mathfrak{w}_c^T \geq (1-\alpha) \gamma^T \big(I - \alpha P_{\min}\big)^{-1}, \label{positive_w_combined_alternative_formulation} \\
& \ \textnormal{\Crefrange{polyhedral_combined_multiplied}{second_order_combined_multiplied}},  \cref{flow_balance_constr_specific_combined}. 
\end{align}
\end{subequations}
\end{lemma}
\begin{proof}
 Let $f \in F_S$ be fixed. For a given feasible vector $\mathfrak{u}$ of \eqref{Inner_opt_with_c_combined}, we define a vector 
$(\mathfrak{w}_c, \mathfrak{z}_c )$ as in \eqref{definition_w_c_combined} and \eqref{definition_z_c_combined}, respectively. 
    By the definition of $P_{\min}$, ${P}_f \geq P_{\min}$.  Since   $I-\alpha P_{\min} $ and $I - \alpha P_f$ are M-matrices, it follows from Theorem 1.8 of  \cite{johnson2011inverse} that $Q_f \geq \big(I - \alpha P_{\min}\big)^{-1}$, thus \eqref{positive_w_combined_alternative_formulation} is  satisfied from \eqref{definition_w_c_combined}. From \Cref{irreducible_assumption}, we obtain $\mathfrak{w}_c > 0$. The remaining arguments follow similarly to the proof of \Cref{lemma_equivalent_LP_problem_specific_combined} under \cref{irreducible_assumption}  instead of \Cref{assumption_positive_gamma}. 
\end{proof}
 Similar to the preceding lemma, SOCP reformulations hold for all other inner optimization problems in \eqref{robust_cmdp_with_stationary_policy}.  
By substituting the inner optimization problems in \eqref{robust_cmdp_with_stationary_policy} with these SOCP problems, we derive an equivalent reformulation of \eqref{robust_cmdp_with_stationary_policy} as a minimization problem.
\begin{theorem}\label{Theorem_duality_bilinear_socp}
    Let \Crefrange{Strict_feasibility_assumption}{irreducible_assumption} hold true. The robust CMDP problem \eqref{robust_cmdp_with_stationary_policy} with $\mathfrak{U}$ defined by \eqref{general_US_combined}, is equivalent to the following problem. 
        \begin{align}\label{robust_cmdp_BP_SOCP_formulation_alternate_reformulation}
      &\min_{z, f, \mathfrak{d}_c, \mathfrak{d}_{d^k}}  \ z \nonumber \\
   \textnormal{s.t.}   & \  (1-\alpha ) \gamma^T \big(\varsigma_c  - \big(I - \alpha P_{\min}\big)^{-1} \eta_c \big) \leq z, \nonumber \\
    & \  (1-\alpha ) \gamma^T \big(\varsigma_{d^k}  - \big(I - \alpha P_{\min}\big)^{-1} \eta_{d^k} \big) \leq \xi_k, \ \forall \ k \in \mathbb{K}, \nonumber \\
    & \ \textnormal{\Crefrange{bilinear_c_constraint}{SOC_c_constraints}},\ \textnormal{\Crefrange{bilinear_dk_constraint}{SOC_dk_constraints}}.   
        \end{align}
\end{theorem}
\begin{proof}
    For a fixed  $z$ and $f \in F_s$, it follows from \Cref{lemma_alternative_equivalent_LP_problem_specific_combined} that the inner optimization problem in the first constraint of \eqref{robust_cmdp_with_stationary_policy} can be equivalently reformulated into the SOCP problem \eqref{equivalent_SOCP_problem_specific_combined_alternative_formulation}. Similar to \Cref{Deriving_dual_appendix}, its dual problem is given by 
    \begin{align*}
  \min_{\mathfrak{d}_c } &\ (1-\alpha ) \gamma^T \big(\varsigma_c  - \big(I - \alpha P_{\min}\big)^{-1} \eta_c \big)\\
  \textnormal{s.t.} &\  \textnormal{\Crefrange{bilinear_c_constraint}{SOC_c_constraints}}.
\end{align*}
Furthermore, it follows from \Cref{Strict_feasibility_assumption} that the optimal values of the above problem and \eqref{equivalent_SOCP_problem_specific_combined_alternative_formulation} are equal. We derive dual problems for all other inner optimization problems of \eqref{robust_cmdp_with_stationary_policy} similarly. Therefore, by substituting these dual problems in \eqref{robust_cmdp_with_stationary_policy}, we obtain \eqref{robust_cmdp_BP_SOCP_formulation_alternate_reformulation}.
\end{proof}

\color{black}


\bibliography{sn-bibliography}

\end{document}